\documentclass[11pt]{article}
\usepackage{amsmath,latexsym,amsxtra,enumerate,
color, cancel}
\usepackage{bbm, dsfont}

\usepackage{amsthm}
\usepackage[english]{babel}
\usepackage[usenames,dvipsnames,svgnames,table]{xcolor}
\usepackage{times, amsfonts, amssymb, amsthm, amscd, graphicx,stmaryrd}
\usepackage[mathscr]{euscript}

\RequirePackage[colorlinks,citecolor=blue,urlcolor=blue]{hyperref}
\hypersetup{colorlinks=true}
\usepackage[active]{srcltx}
\usepackage{a4wide}

\theoremstyle{plain}
\newtheorem{thm}{Theorem}[section]
\newtheorem{lemma}[thm]{Lemma}

\newtheorem{cor}[thm]{Corollary}

\theoremstyle{definition}

\newtheorem{definition}[thm]{Definition}

\newtheorem{rem}[thm]{Remark}

\allowdisplaybreaks

\newcommand{\equa}{\begin{eqnarray*}}
\newcommand{\tion}{\end{eqnarray*}}
\newcommand{\equal}{\begin{eqnarray}}
\newcommand{\tionl}{\end{eqnarray}}

\makeatletter
\def\timenow{\@tempcnta\time
\@tempcntb\@tempcnta
\divide\@tempcntb60
\ifnum10>\@tempcntb0\fi\number\@tempcntb
:\multiply\@tempcntb60
\advance\@tempcnta-\@tempcntb
\ifnum10>\@tempcnta0\fi\number\@tempcnta}
\makeatother

\title{ (Non-)Distributivity of the Product for $\sigma$-Algebras with Respect to the Intersection}

\author{ Alexander Steinicke$^1$\\
\small  February 22, 2021
}
\date{}
\begin{document}

\maketitle
\begin{abstract}
We study the validity of the distributivity equation
$$(\mathcal{A}\otimes\mathcal{F})\cap(\mathcal{A}\otimes\mathcal{G})=\mathcal{A}\otimes\left(\mathcal{F}\cap\mathcal{G}\right),$$
where $\mathcal{A}$ is a $\sigma$-algebra on a set $X$, and $\mathcal{F}, \mathcal{G}$ are $\sigma$-algebras on a set $U$. \\
We present a counterexample for the general case and in the case of countably generated subspaces of analytic measurable spaces we give an equivalent condition in terms of the $\sigma$-algebras' atoms. Using this, we give a sufficient condition under which distributivity holds.
\end{abstract}

\vspace{1em}
{\noindent \textit{Keywords:} sigma algebra; intersection of sigma algebras; product sigma algebras; counterexample for sigma algebras
}
{\noindent
\footnotetext[1]{Department of Mathematics and Information Technology, Montanuniversitaet Leoben, Austria. \\ \hspace*{1.5em}  {\tt alexander.steinicke{\rm@}unileoben.ac.at}}}


\section{Introduction}

Since the axiomatisation of probability theory or -- more generally -- measure theory, the most fundamental structure of these theories are $\sigma$-algebras (often called $\sigma$-fields or also Borel structures). Together with similar structures, such as $\sigma$-rings or set algebras, they are well investigated as can be seen by opening any book on measure theory. These objects also appear in the context of boolean algebras (see e.g. the Stone representation theorem or the Loomis-Sikorski theorem). However, large amounts of general results are widely spread through the literature and an overall survey of the general results for $\sigma$-algebras is not known to the author. Indeed, collecting all such results was called a Herculanean task by Bhaskara Rao and Rao \cite{RaoRao} in their 1981 article 'Borel spaces'. We further want to mention the works of Aumann \cite{Aumann}, Basu \cite{Basu}, Bhaskara Rao and Bhaskara Rao \cite{BhaskaraBhaskara}, Blackwell \cite{Blackwell}, Georgiou \cite{Georgiou}, Grzegorek \cite{Grzegorek}, who contributed to the theory beyond the scope of a book on measure theory or a related topic. 

Within the subject of $\sigma$-algebras, even very simple constructions may already lead to nontrivial questions such as the following:

Let $\mathcal{A}$ be a $\sigma$-algebra on a set $X$ and $\mathcal{F},\mathcal{G}$ be two $\sigma$-algebras on a set $U$. The product of the $\sigma$-algebras $\mathcal{A}\otimes\mathcal{F}$ on $X\times U$ is defined as usual, as the smallest $\sigma$-algebra containing all Cartesian products (or rectangles) $\{A\times F: A\in \mathcal{A}, F\in \mathcal{F}\}$. We ask, for which $\sigma$-algebras $\mathcal{A},\mathcal{F}, \mathcal{G}$ is
\begin{align}\label{eq:main}
(\mathcal{A}\otimes\mathcal{F})\cap(\mathcal{A}\otimes\mathcal{G})=\mathcal{A}\otimes\left(\mathcal{F}\cap\mathcal{G}\right),
\end{align}
which we can interpret as distributivity of '$\otimes$', with respect to '$\cap$'. 
The problem posed in \eqref{eq:main}, consists of a simple binary composition of fundamental objects in measure theory and is therefore of its own interest. While the problem as such has not been posed before, some results have been derived as part of other studies and are summarized in Section \ref{counter}. To us, it was motivated by a question arising in stochastic analysis: let $f,g$ be two random variables on a probability space $(\Omega,\mathcal{F},\mathbb{P})$ (in the original question, the random variables were c\`adl\`ag processes). Consider then all random functions that may be expressed by a Borel function of $f$ and simultaneously by a function of $g$ and additionally have a measurable dependence on a real parameter. Such functions are contained in the set of all functions $h\colon (\Omega\times{[0,1]},(\sigma(f)\cap\sigma(g))\otimes\mathcal{B}([0,1]))\to\mathbb{R}$ (here $\sigma(f)$ denotes the $\sigma$-algebra generated by $f$, the same for $g$). Is this set of functions the same as the one consisting of functions, which may be expressed in two ways: 1) as bivariate Borel function applied to $f$ and the real parameter, and 2) as another bivariate Borel function applied to $g$ and the real parameter? The latter set would correspond to the functions measurable with respect to $$\left(\sigma(f)\otimes\mathcal{B}([0,1])\right)\cap\left(\sigma(g)\otimes\mathcal{B}([0,1])\right).$$ 

Regarding the general problem \eqref{eq:main} there is one immediate observation: since the right hand side is contained in both $\sigma$-algebras taking part in the intersection on the left, the left hand side is larger.
\begin{align*}
(\mathcal{A}\otimes\mathcal{F})\cap(\mathcal{A}\otimes\mathcal{G})\supseteq\mathcal{A}\otimes\left(\mathcal{F}\cap\mathcal{G}\right).
\end{align*}
For most simple examples of $\sigma$-algebras, e.g. Borel $\sigma$-algebras on a metric space and their restrictions on subsets or $\sigma$-algebras of symmetric or periodic sets, the converse inclusion is also satisfied. It is also easy to see that equality holds if $\mathcal{F}\subseteq\mathcal{G}$ or $\mathcal{G}\subseteq\mathcal{F}$.

It is tempting to assume, that relation \eqref{eq:main} should be as easy to prove or disprove, as the similar relation for the supremum of two $\sigma$-algebras $\mathcal{F}\vee\mathcal{G}$, which is the smallest $\sigma$-algebra containing $\mathcal{F}\cup\mathcal{G}$. It is not hard to show that
\begin{align*}
(\mathcal{A}\otimes\mathcal{F})\vee(\mathcal{A}\otimes\mathcal{G})=\mathcal{A}\otimes\left(\mathcal{F}\vee\mathcal{G}\right),
\end{align*}
see e.g. \cite[Proof of Lemma 3.2, Step 2]{Steinicke}. Alas, in the case of the intersection (or infimum) of $\sigma$-algebras there are no obvious reasons why the equality should hold. To show that the intersection might be the source of trouble, consider the following equation 
$$\left(\mathcal{A}\vee\mathcal{F}\right)\cap\left(\mathcal{A}\vee\mathcal{G}\right)=\mathcal{A}\vee\left(\mathcal{F}\cap\mathcal{G}\right).$$
Examples violating this equation can already be found for finite sub-$\sigma$-algebras of a $\sigma$-algebra that has cardinality strictly larger than $4$, see \cite[Proposition 34]{RaoRao}, where the lattice of $\sigma$-algebras on a set is studied. Equation \eqref{eq:main} does not admit such simple counterexamples.

Another indicator, why equation \eqref{eq:main} could be wrong in general is the following expression appearing in ergodic theory \cite{Parry}. For a decreasing sequence of sub $\sigma$-algebras $\left(\mathcal{A}_n\right)_{n\in \mathbb{N}}$ of $\mathcal{A}$ on a set $X$, and a measurable space $(U,\mathcal{U})$, the inequality
\begin{align*}
\left(\bigcap_{n\in \mathbb{N}}\mathcal{A}_n\right)\otimes\mathcal{U}\subseteq \bigcap_{n\in \mathbb{N}}\left(\mathcal{A}_n\otimes\mathcal{U}\right)
\end{align*}
can be strict. The difference between both sets is subtle: whenever we have probability measures $\mu$ and $\nu$ on the measure spaces $(X,\mathcal{A},\mu), (U,\mathcal{U},\nu)$, all sets $K$ in
$$\bigcap_{n\in \mathbb{N}}\left(\mathcal{A}_n\otimes\mathcal{U}\right)\setminus \left(\bigcap_{n\in \mathbb{N}}\mathcal{A}_n\right)\otimes\mathcal{U}$$
satisfy $(\mu\otimes\nu)(K)=0$ and are thus null sets for every arbitrary pair of probability measures $(\mu,\nu)$, see e.g. \cite[Section 4, Exercise 6]{Parry} and its discussion in \cite{stackexchange}.\bigskip

In Section \ref{counter} we will elaborate on a counterexample to \eqref{eq:main}, which is credited to G. Halmos and appeared first in Aumann \cite{Aumann}. We will also derive an equivalent condition to \eqref{eq:main} and use it to show that an intersection of products may not be the product of any $\sigma$-algebras on the same sets.

Section \ref{char} contains conditions under which equality \eqref{eq:main} is satisfied and which explain the validity of the equation for countably generated examples. For a certain class of $\sigma$-algebras we will formulate another condition in terms of the $\sigma$-algebras' atoms and show that this condition is both, sufficient and necessary.

\section{A First Negative Answer}\label{counter}
The counterexample in this section first appeared in a different context in the work of Aumann \cite{Aumann} who refers it to G. Halmos. In order to state and understand it, we need the following notation:
\begin{itemize}
\item For a set $X$, we denote the power set of $X$ by $\mathcal{P}(X)$.
\item Let $\mathcal{C}$ denote the $\sigma$-algebra of countable and co-countable sets on the unit interval $[0,1]$.

\item Let $\mathcal{B}$ denote the Borel $\sigma$-Algebra on $[0,1]$ with respect to the usual topology $\mathcal{T}$ on ${[0,1]}$.

\item The $\sigma$-algebra $\mathcal{D}$ is the preimage of $\mathcal{B}$ under a certain function $f\colon{[0,1]}\to{[0,1]}$ constructed  as follows (see also \cite{Aumann}, \cite{Rao}):

Let $\omega_\mathfrak{c}$ be the first ordinal corresponding to the cardinal $\mathfrak{c}$ of the continuum. Let $(M_{\alpha})_{1\leq\alpha<\omega_{\mathfrak{c}}}$ be an enumeration of all uncountable Borel subsets of $[0,1]$ with uncountable complement. Since all uncountable Borel sets have cardinality $\mathfrak{c}$ (see e.g. \cite[Theorem 13.6]{Kechris}) we can associate to each ordinal $\alpha<\omega_\mathfrak{c}$ a triplet $(x_\alpha,y_\alpha,z_\alpha)$ such that $x_\alpha, y_\alpha\in M_\alpha$, $z_\alpha\in [0,1]\setminus M_\alpha$ and $\{x_\alpha,y_\alpha,z_\alpha\}\cap \bigcup_{\beta<\alpha}\left\{x_\beta,y_\beta,z_\beta\right\}=\emptyset$ (as the cardinality of $\bigcup_{\beta<\alpha}\left\{x_\beta,y_\beta,z_\beta\right\}$ is strictly smaller than $\mathfrak{c}$).
Define $f$ as the function that for each $\alpha<\omega_\mathfrak{c}$ maps $x_\alpha\mapsto z_\alpha$, $z_\alpha\mapsto x_\alpha$ and keeps $y_\alpha$ and all points outside $\bigcup_{\beta<\omega_\mathfrak{c}}\left\{x_\beta,y_\beta,z_\beta\right\}$ fixed. 

Finally, set \begin{align*}
\mathcal{D}:=\sigma(f)=\left\{f^{-1}(B):B\in \mathcal{B}\right\}.
\end{align*}
\end{itemize}
The following lemma (found in Rao \cite{Rao}) shows a relation between the $\sigma$-algebras defined above.
\begin{lemma}[{\cite[Theorem 1]{Rao}}]\label{lem:intersection}
The intersection $\mathcal{B}\cap\mathcal{D}$ equals the $\sigma$-algebra $\mathcal{C}$ of countable and co-countable sets (which is not countably generated).
\end{lemma}
\begin{proof}
For each $\alpha<\omega_\mathfrak{c}$, $M_\alpha$ is changed by $f$ since the $x_\alpha$ are changed to $z_\alpha$. Assume that the set $F_\alpha=\{x\in M_\alpha: f(x)=x\}$ is at most countable. Then $M_\alpha\setminus F_\alpha$ remains uncountable as does its complement $[0,1]\setminus(M_\alpha\setminus F_\alpha)$. Thus, $M_\alpha\setminus F_\alpha=M_\beta$ for some ordinal $\beta<\omega_\mathfrak{c}$. As $y_\beta$ is kept fixed by $f$, it is a fixed point not included in $F_\alpha$, which is a contradiction. Thus an uncountable number of points in each $M_\alpha$ are kept fixed by $f$. This implies that, the set $f^{-1}(M_\alpha)\cap M_\alpha$, which is invariant under $f$, is uncountable. Its complement is also uncountable since it contains the complement of $M_\alpha$. If $f^{-1}(M_\alpha)$ were Borel, the set $f^{-1}(M_\alpha)\cap M_\alpha$ is too as intersection of Borel sets. Being an uncountable Borel set with uncountable complement, $f^{-1}(M_\alpha)\cap M_\alpha=M_\gamma$ for some $\gamma<\omega_\mathfrak{c}$, which leads to a contradiction since $M_\gamma$ cannot be fixed under $f$. Thus $f^{-1}(M_\alpha)$ is not a Borel set and the only Borel sets in $\mathcal{D}$ are the countable and co-countable ones.
\end{proof}

The counterexample is now essentially a combination of Lemmas 7.1, 7.2 in \cite{Aumann}. For the reader's convenience we provide a full proof.
\begin{thm}[{}]\label{thm:counterexample}
Let $\mathcal{D}$ and $\mathcal{B}$ be the $\sigma$-algebras defined above. Then
\begin{align*}
\left((\mathcal{D}\vee\mathcal{B})\otimes\mathcal{B}\right)\cap\left((\mathcal{D}\vee\mathcal{B})\otimes\mathcal{D}\right)\neq (\mathcal{D}\vee\mathcal{B})\otimes(\mathcal{B}\cap \mathcal{D}).
\end{align*}
\end{thm}
\begin{proof}
Let $\Delta:=\{(x,x)\in {[0,1]}^2\colon x\in [0,1]\}$ be the diagonal, which as a closed set is contained in $\mathcal{B}\otimes\mathcal{B}$. The $\sigma$-algebra $\mathcal{D}$ can be seen as the Borel $\sigma$-algebra generated by the topology $$\mathcal{T}_f:=\left\{f^{-1}(O): O\in \mathcal{T}\right\},$$
where $f$ is the function defined above Lemma \ref{lem:intersection}. Since $f$ is bijective, we know that $([0,1],\mathcal{T}_f)$ is Hausdorff. Therefore $\Delta$ is a closed set with respect to the product topology of $([0,1],\mathcal{T}_f)^2$ (this follows e.g. from \cite[Theorem 1.5.4]{Engelking} setting $X=[0,1]^2$, using the projection functions onto the first and onto the second variable). Since $([0,1],\mathcal{T}_f)$ is a separable metric space, by \cite[Lemma 6.4.2]{Bogachev} the $\sigma$-algebra generated by the product topology of $([0,1],\mathcal{T}_f)^2$ is $\mathcal{D}\otimes\mathcal{D}$. Hence $\Delta$ is contained in $(\mathcal{D}\otimes\mathcal{D})\cap (\mathcal{B}\otimes\mathcal{B})$ and is thus also contained in $$
\left((\mathcal{D}\vee\mathcal{B})\otimes\mathcal{B}\right)\cap\left((\mathcal{D}\vee\mathcal{B})\otimes\mathcal{D}\right).$$

It is left to show that $\Delta\notin (\mathcal{D}\vee\mathcal{B})\otimes(\mathcal{B}\cap \mathcal{D})$. As seen before in Lemma \ref{lem:intersection}, $\mathcal{B}\cap \mathcal{D}=\mathcal{C}$ is the $\sigma$-algebra of countable and co-countable sets. In fact, $\Delta$ is not even contained in the larger $\sigma$-algebra $\mathcal{P}([0,1])\otimes\mathcal{C}$ as can be seen by the following argument:

Let $\Pi:=\left\{A\times C : A\in \mathcal{P}([0,1]), C\in \mathcal{C}\right\}$ be the $\pi$-system of Cartesian products of the sets in $\mathcal{P}([0,1])$ and $\mathcal{C}$. Let $\mathscr{H}$ denote the set of functions 
\begin{align*}
\mathscr{H}:=\bigg\{f\colon{[0,1]}^2\to\mathbb{R} : \exists g\colon{[0,1]}\to\mathbb{R} \text{ s.t. } f(x,y)=g(x) \text{ except for countably many } y\bigg\}.
\end{align*}
It is easy to check that $\mathscr{H}$ is a vector space. For the characteristic function $\chi_{A\times C}$ with $A\times C\in \Pi$, the mapping $y\mapsto \chi_{A\times C}(\cdot, y)$ equals the $0$ function for all but countably many $y$ if $C$ is countable, and equals
 the function $\chi_A$ for all but countably many $y$ if $C$ is co-countable. Hence $\chi_{A\times C}\in \mathscr{H}$.
 Consider a monotone limit of functions $f_n$ in $\mathscr{H}$ converging to $f\colon{[0,1]}^2\to\mathbb{R}$. For all $n\in \mathbb{N}$ there is a countable set $A_n$ such that for some function $g_n\colon{[0,1]}\to\mathbb{R}$ and for all $y\in {[0,1]}\setminus \!A_n$ we have $f_n(x,y)=g_n(x)$. Since 
 \begin{align*}
 f_n\nearrow f \text{ for }y\in {[0,1]}\setminus\left(\bigcup_{n\in\mathbb{N}}A_n\right),
 \end{align*} 
 it follows that $f_n(x,y)=g_n(x)\nearrow f(x,y)$. Hence on this set, the limit $f$ does not depend on $y$ and is thus a function $g\colon{[0,1]}\to\mathbb{R}$ for all $y$ except those that are in the at most countable set $\bigcup_{n\in\mathbb{N}}A_n$. Therefore, $f$ is also contained in $\mathscr{H}$. We infer now by the monotone class theorem for functions that $\mathscr{H}$ contains all bounded $\mathcal{P}({[0,1]})\otimes\mathcal{C}$-measurable functions, in particular all indicator functions of the $\sigma$-algebra's elements. We next show that $\chi_\Delta$ is not included in $\mathscr{H}$: the mapping $y\mapsto \chi_\Delta(\cdot,y)$ has a different value for each $y\in {[0,1]}$, namely $\chi_{\{y\}}$, thus does not equal one function $x\mapsto g(x)$ for all but countably many $y$.
 Hence, $\chi_\Delta\notin\mathscr{H}$. It follows that $\Delta\notin \mathcal{P}([0,1])\otimes\mathcal{C}$ and is thus neither in $(\mathcal{D}\vee\mathcal{B})\otimes\mathcal{C}$ which proves the assertion.
 \end{proof}
 
 \begin{rem}
 The construction of the above $\sigma$-algebra $\mathcal{D}$ relies on the function $f$. The same construction can be directly transferred to the context of topologies: Let $\mathcal{T}_f$ denote the initial topology of $f$ (as in the above proof) and $\mathcal{T}$ the usual topology on ${[0,1]}$. Then $\mathcal{T}_f\cap\mathcal{T}$ is the topology of co-countable subsets of ${[0,1]}$. With the notation $\vee$ for the supremum and $\otimes_{\operatorname{Top}}$ for the product topology of two given topologies, by the same procedure as in the above proof we see that
 \begin{align*}
 {[0,1]}^2\setminus\Delta\in \left((\mathcal{T}_f\vee\mathcal{T})\otimes_{\operatorname{Top}}\mathcal{T}\right)\cap \left((\mathcal{T}_f\vee\mathcal{T})\otimes_{\operatorname{Top}}\mathcal{T}_f\right),
 \end{align*}
 but
 \begin{align*}
 {[0,1]}^2\setminus\Delta\notin (\mathcal{T}_f\vee\mathcal{T})\otimes_{\operatorname{Top}}(\mathcal{T}\cap\mathcal{T}_f).
 \end{align*}
 \end{rem}

The following characterization will lead to a strengthening of the assertion of Theorem \ref{thm:counterexample}. It will show that the sets are not only different but that the intersection of the product $\sigma$-algebras cannot even be written as product of any other sub-$\sigma$-algebras of ${[0,1]}$.

\begin{lemma}
For $\sigma$-algebras $\mathcal{A}$ on a set $X$ and $\mathcal{F},\mathcal{G}$ on $U$ the following assertions are equivalent:
\begin{enumerate}[(i)]
\item\label{uno} $(\mathcal{A}\otimes\mathcal{F})\cap(\mathcal{A}\otimes\mathcal{G})=\mathcal{A}\otimes\left(\mathcal{F}\cap\mathcal{G}\right)$
\item\label{due} There are $\sigma$-algebras $\mathcal{A}_0$ on $X$ and $\mathcal{E}$ on $U$ such that 
\begin{align*}
(\mathcal{A}\otimes\mathcal{F})\cap(\mathcal{A}\otimes\mathcal{G})=\mathcal{A}_0\otimes\mathcal{E}.
\end{align*}
\end{enumerate}
\end{lemma}
\begin{proof}
The direction \eqref{uno}$\Rightarrow$\eqref{due} is obvious.

For the other implication assume that $(\mathcal{A}\otimes\mathcal{F})\cap(\mathcal{A}\otimes\mathcal{G})$ is of product form $\mathcal{A}_0\otimes\mathcal{E}$. Clearly, for all $A\in\mathcal{A}_0$ and $E\in\mathcal{E}$ we have $A\times E\in \mathcal{A}_0\otimes\mathcal{E}$. Since also $A\times E\in \mathcal{A}\otimes\mathcal{F}$ and since sections are measurable, it follows that $E\in \mathcal{F}$ and $A\in \mathcal{A}$. In the same way, since $A\times E\in \mathcal{A}\otimes\mathcal{G}$, we can conclude that $E\in \mathcal{G}$. Hence, $A\times E\in \mathcal{A}\otimes\left(\mathcal{F}\cap\mathcal{G}\right)$. Since $\mathcal{A}_0\otimes\mathcal{E}$ is generated by products, which are all contained in $\mathcal{A}\otimes\left(\mathcal{F}\cap\mathcal{G}\right)$, 
$$\mathcal{A}_0\otimes\mathcal{E}\subseteq \mathcal{A}\otimes\left(\mathcal{F}\cap\mathcal{G}\right),$$
which implies \eqref{uno}.
\end{proof}
The lemma leads to the following corollary.
\begin{cor}
The $\sigma$-algebra $\left((\mathcal{D}\vee\mathcal{B})\otimes\mathcal{B}\right)\cap\left((\mathcal{D}\vee\mathcal{B})\otimes\mathcal{D}\right)$ is not of product form $\mathcal{A}_0\otimes\mathcal{E}$ for any $\sigma$-algebras $\mathcal{A}_0,\mathcal{E}$ on ${[0,1]}^2$.
\end{cor}


\section{Characterization of Distributivity in the Case of Countable Generation}\label{char}

In this section we derive a positive result for the relation 
\begin{align*}
(\mathcal{A}\otimes\mathcal{F})\cap(\mathcal{A}\otimes\mathcal{G})=\mathcal{A}\otimes(\mathcal{F}\cap\mathcal{G}).
\end{align*}
To that end, we give a characterization of the above equation if all involved $\sigma$-algebras are countably generated sub $\sigma$-algebras of analytic spaces (see Definition \ref{def:count} below). 

\begin{definition}\label{def:count}\noindent
Let $(X,\mathcal{M})$ be a measurable space. Recall that a subset of $[0,1]$ is called {\it analytic} if it is the image of a Polish space under a continuous function.
\begin{enumerate}[(i)]
\item  We call $\mathcal{M}$ {\it countably generated} if it is generated by an at most countable family of sets, i.e. $\mathcal{M}=\sigma\left(\{M_n, n\in \mathbb{N}\}\right)$.

\item We call $\mathcal{M}$ {\it countably separated} if there is an at most countable collection of sets $(M_n)_{n\in\mathbb{N}}$ that separates the points of $X$, i.e. for all $x,y\in X, x\neq y$ there is $n\in\mathbb{N}$ such that $x\in M_n$ and $y\notin M_n$ or $x\notin M_n$ and $y\in M_n$. 

\item We call $(X,\mathcal{M})$ {\it analytic} if $X$ is isomorphic (i.e. there is a bijective, bimeasurable function) to an analytic subset of the unit interval and $\mathcal{M}$ is countably generated and contains all singletons of $X$.

\item An {\it atom} of $\mathcal{M}$ is a nonempty set $A\in\mathcal{M}$ such that no proper, nonempty subset of $A$ is contained in $\mathcal{M}$. 
\end{enumerate}

\end{definition}

In order to derive the equivalent condition for \eqref{eq:main} we need to present a result of Blackwell \cite{Blackwell} and Mackey \cite{Mackey} first, which we will cite as Lemma \ref{lem:blackwell} below, see also Bhaskara Rao and Rao \cite{RaoRao}.

\begin{lemma}[{\cite[Section 4]{Blackwell}, \cite[Section 4]{Mackey}, \cite[Proposition 6]{RaoRao}}]\label{lem:blackwell}\ \\
If $(X,\mathcal{M})$ is an analytic space and $\mathcal{F},\mathcal{G}$ are countably generated sub $\sigma$-algebras of $\mathcal{M}$ with the same atoms, then $\mathcal{F}=\mathcal{G}$. 
\end{lemma}

With this lemma at hand, we show the following

\begin{thm}\label{thm:productatoms}
Let $(X,\mathcal{M}),(U,\mathcal{U})$ be analytic spaces and let $\mathcal{A}\subseteq\mathcal{M}$ and $\mathcal{F},\mathcal{G}\subseteq\mathcal{U}$ be $\sigma$-algebras on $X$ and $U$, respectively. Assume further that $\mathcal{A}, \mathcal{F}\cap\mathcal{G}$ and $(\mathcal{A}\otimes\mathcal{F})\cap(\mathcal{A}\otimes\mathcal{G})$ are countably generated. The following assertions are equivalent:
\begin{enumerate}[(i)]
\item\label{i} $(\mathcal{A}\otimes\mathcal{F})\cap(\mathcal{A}\otimes\mathcal{G})=\mathcal{A}\otimes(\mathcal{F}\cap\mathcal{G})$
\item\label{ii} all atoms of $(\mathcal{A}\otimes\mathcal{F})\cap(\mathcal{A}\otimes\mathcal{G})$ are Cartesian products.
\end{enumerate}
\end{thm}
\begin{proof}
Let $K$ be an atom of $\mathcal{A}\otimes(\mathcal{F}\cap\mathcal{G})$. Let $\left(A_n\right)_{n\in\mathbb{N}}$ and $\left(C_n\right)_{n\in\mathbb{N}}$ be generators of $\mathcal{A}$ and $\mathcal{F}\cap\mathcal{G}$. Then $\left(A_n\times C_m\right)_{(n,m)\in\mathbb{N}^2}$ is a countable generator of $\mathcal{A}\otimes(\mathcal{F}\cap\mathcal{G})$. Let $(x,u)\in K$. Define 
\begin{align*}
V:=\bigcap_{(n,m)\in\mathbb{N}^2} V_{n,m},\quad \text{where }V_{n,m}=A_n\times C_m \text{ or } V_{n,m}=(X\times U)\setminus(A_n\times C_m)
\end{align*}
according to whether $(x,u)\in A_n\times C_m$ or $(x,u)\in (X\times U)\setminus(A_n\times C_m)$. Then $V$ is an atom in $\mathcal{A}\otimes(\mathcal{F}\cap\mathcal{G})$ that contains $(x,u)$ (see \cite[Section 1]{RaoRao}) and therefore must be $K$. As all $V_{m,n}$ are either Cartesian products or finite unions of Cartesian products, $V$ must also be an at most countable union of Cartesian products. But if $V$ were a proper union of sets in $\mathcal{A}\otimes(\mathcal{F}\cap\mathcal{G})$, it cannot be an atom. Hence $V=K$ has a product form $K=A\times C$. 

If now \eqref{i} is satisfied, the above implies that all atoms of $(\mathcal{A}\otimes\mathcal{F})\cap(\mathcal{A}\otimes\mathcal{G})$ are Cartesian products, implying \eqref{ii}.

On the other hand, assume \eqref{ii}. Then every atom $L$ in $(\mathcal{A}\otimes\mathcal{F})\cap(\mathcal{A}\otimes\mathcal{G})$ is a product, $L=B\times D$. As sections are measurable, it follows that $D\in \mathcal{F}$ and $D\in \mathcal{G}$, so $D\in \mathcal{F}\cap\mathcal{G}$. Since 
\begin{align*}
(\mathcal{A}\otimes\mathcal{F})\cap(\mathcal{A}\otimes\mathcal{G})\supseteq\mathcal{A}\otimes(\mathcal{F}\cap\mathcal{G}),
\end{align*}
$L$ is also an atom in $\mathcal{A}\otimes(\mathcal{F}\cap\mathcal{G})$. Conversely, assume that $L$ is an atom in $\mathcal{A}\otimes(\mathcal{F}\cap\mathcal{G})$. As $(\mathcal{A}\otimes\mathcal{F})\cap(\mathcal{A}\otimes\mathcal{G})$ is countably generated and $L$ is contained therein, $L$ is a union of atoms in $(\mathcal{A}\otimes\mathcal{F})\cap(\mathcal{A}\otimes\mathcal{G})$ (see \cite[Section 1]{RaoRao}). If this union were proper, there were another atom $\check L\neq\emptyset$ in $(\mathcal{A}\otimes\mathcal{F})\cap(\mathcal{A}\otimes\mathcal{G})$ such that $\check L\subsetneqq L$. Then, following the above argument, $\check L$ were contained in $\mathcal{A}\otimes(\mathcal{F}\cap\mathcal{G})$, a contradiction to $L$ being an atom. Hence, both sets in \eqref{i} have the same atoms in $\mathcal{A}\otimes(\mathcal{F}\cap\mathcal{G})$. As $(X\times U,\mathcal{M}\otimes\mathcal{U})$ is analytic as well (see \cite[Lemma 6.6.5 (iii)]{Bogachev}), by Lemma \ref{lem:blackwell} both sets in \eqref{i} are equal.
\end{proof}

We will now show that the assertions of the equivalence in the above Theorem hold if the intersection of products is countably separated. We need the following result taken from \cite[Theorem 6.5.8]{Bogachev}.

\begin{lemma}\label{lem:cntgenfunction}
Let $(E,\mathcal{E})$ be a measurable space. Then $\mathcal{E}$ is countably generated and countably separated if and only if $(E,\mathcal{E})$ is isomorphic to a subset $A$ in $[0,1]$ with the induced Borel $\sigma$-algebra, i.e. there is a bijective, $\mathcal{E}-\mathcal{B}$ measurable mapping $H\colon E\to A$ such that
$$\mathcal{E}=\sigma(H)=\left\{H^{-1}(B): B\in \mathcal{B}(A)\right\},$$
where $\mathcal{B}(A):=\left\{B\cap A: B\in\mathcal{B}\right\}.$
\end{lemma}

\begin{rem}
\begin{enumerate}
\item If $\mathcal{F}=\sigma(f)$ and $\mathcal{G}=\sigma(g)$ for functions $g, f\colon U\to{[0,1]}$ then $\mathcal{F}\cap\mathcal{G}$ is countably generated if and only if there is a function $h\colon U\to{[0,1]}$ such that $\mathcal{F}\cap\mathcal{G}=\sigma(h)$ (see \cite[Theorem 6.5.5]{Bogachev}).
\item If $\mathcal{F}\cap\mathcal{G}$ is countably separated, then, in addition, $h$ can be chosen injective (this is the assertion of the previous Lemma).
\end{enumerate}
\end{rem}
We conclude this section by showing that equality in \eqref{eq:main} holds if the intersection of products is countably separated.
\begin{thm}
Let $\mathcal{A}, \mathcal{F}, \mathcal{G}$ and $\mathcal{F}\cap\mathcal{G}$ be as in Theorem \ref{thm:productatoms}. Moreover, let $\left(\mathcal{A}\otimes\mathcal{F}\right)\cap\left(\mathcal{A}\otimes\mathcal{G}\right)$ be countably generated and countably separated. Then, \eqref{eq:main} is satisfied.
\end{thm}
\begin{proof}
By Lemma \ref{lem:cntgenfunction}, we know that $\left(\mathcal{A}\otimes\mathcal{F}\right)\cap\left(\mathcal{A}\otimes\mathcal{G}\right)=\sigma(H)$ for an appropriate function $H$. Further, the lemma implies that $$\sigma(H)=\left\{H^{-1}(B): B\in\mathcal{B}(H(X\times U)) \right\}$$ with $\mathcal{B}(H(X\times U))$ denoting the induced Borel $\sigma$-algebra on $H(X\times U)$. As the singletons of $H(X\times U)$ are contained in $\mathcal{B}(H(X\times U))$, by the injectivity of $H$, the atoms of $\sigma(H)$ are singletons as well and thus of product form. The assertion now follows from Theorem \ref{thm:productatoms}.
\end{proof}

\section*{Acknowledgements}
The author wants to thank Karin Schnass, University of Innsbruck, and Gunther Leobacher, University of Graz, for fruitful discussions and valuable suggestions.

\bibliographystyle{plain}

\begin{thebibliography}{99}
\bibitem{Aumann} R. Aumann, {\it Borel structures for function spaces}, Ill. J. Math. 5, p. 614-635, 1961.
\bibitem{Basu} D. Basu, {\it Problems relating to the existence of maximal and minimal elements in some families of statistics (subfields)}, in Proceedings of the Fifth Berkeley Symposium on Mathematical Statistics and Probability, Volume 1: Statistics, p. 41-50, 1967.
\bibitem{BhaskaraBhaskara} M. Bhaskara Rao, K.P.S. Bhaskara Rao, {\it Borel $\sigma$-algebra on ${[0,\Omega]}$}, Bull. Acad. Polon. Sci. 26, p. 767-769, 1978.
\bibitem{RaoRao} K.P.S. Bhaskara Rao, B.V. Rao, {\it Borel Spaces}, Warszawa: Instytut Matematyczny Polskiej Akademi Nauk, http://eudml.org/doc/268562, 1981.
\bibitem{Blackwell} D. Blackwell, {\it On a Class of Probability spaces}, 3rd Berkeley Symposium 2, p. 1, 1956.
\bibitem{Bogachev} V. I. Bogachev, {\it Measure Theory Volume II}, Springer, Berlin, 2007.
\bibitem{Engelking} R. Engelking, {\it General Topology}, Heldermann, Berlin, 1989.
\bibitem{Georgiou} D. Georgiou, I. Kougias, A. Megaritis, {\it Borel Structures on the Set of Borel Mappings}, Topology and its Applications, 159, 7, p. 1906-1915, 2012.
\bibitem{Grzegorek} E. Grzegorek, {\it Remark on Some Borel Structures}, Measure Theory Oberwolfach,  p. 69-74, 1983. 
\bibitem{Kechris} A. Kechris, {\it Classical Descriptive Set Theory}, Springer, New York, 1994.
\bibitem{Mackey} G.W. Mackey, {Borel Structures in Groups and Their Duals}, Transactions of the American Mathematical Society 85, p. 134-165, 1957.
\bibitem{Parry} W. Parry, {\it Topics in Ergodics Theory}, Cambridge University Press, 1981.
\bibitem{Rao} B.V. Rao, {\it On Borel Structures}, Colloq. Math. 21 (1970), 199–204.
\bibitem{stackexchange} Mathematics Stack Exchange, https://math.stackexchange.com/questions/92546/decreasing-sequence-of-product-sigma-algebras,
 viewed on July 10, 2020.
\bibitem{Steinicke} A. Steinicke, {\it Functionals of a L\'evy Process on Canonical and Generic Probability Spaces},
Journal of Theoretical Probability, 29, 2, p. 443-458, 2016.
\end{thebibliography}

\end{document}